\documentclass[letterpaper, 10 pt, conference]{ieeeconf}
\IEEEoverridecommandlockouts       
\usepackage{hyperref}

\usepackage{graphicx}
\usepackage{amsmath,mathtools} 
\usepackage{amssymb, amsmath}
\usepackage{caption}
\usepackage{subfig}
\usepackage{tikz}
\usetikzlibrary{positioning}
\usetikzlibrary{shapes,arrows}
\usepackage{epstopdf}

\newcommand{\lls}{\langle\langle}
\newcommand{\ggs}{\rangle\rangle}

\newtheorem{theorem}{Theorem}

\newtheorem{assumption}{Assumption}

\title{Attitude and Gyro Bias Estimation Using GPS and IMU Measurements}
\author{Soulaimane Berkane and Abdelhamid Tayebi
\thanks{This work was supported by the National Sciences and Engineering Research Council of Canada (NSERC). S. Berkane ({\tt\small sberkane@uwo.ca}) and A. Tayebi ({\tt\small atayebi@lakeheadu.ca}) are both with the Department of Electrical and Computer Engineering, University of Western Ontario, London, Ontario, Canada. A. Tayebi is also with the Department of Electrical Engineering, Lakehead University, Thunder Bay, Ontario, Canada.
}}
\begin{document}
\maketitle
\begin{abstract}
We propose an attitude and gyro-bias estimation scheme for accelerated rigid body systems using an inertial measurement unit (IMU) and a global positioning system (GPS). The proposed scheme allows to obtain attitude estimates directly on the Special Orthogonal group $SO(3)$ while estimating the gyro bias and the unknown apparent acceleration of the vehicle. We prove semi-global exponential stability of the estimation errors. Furthermore, a new switching technique for the attitude state is introduced which results in a velocity-aided hybrid attitude observer with proven global exponential stability.

\end{abstract}
\section{Introduction}
There has been a growing interest in the last decade for the control of Unmanned Aerial Vehicles (UAV), Unmanned Underwater Vehicles (UUV) and other vehicles that operate without human occupant. Attitude information needs to be extracted from on-board sensors for these vehicles to be operated. When cost and size of the vehicle is important, low-cost and small-size sensors are often used in Inertial Navigation Systems (INS). A typical IMU device contains an accelerometer, a gyroscope and a magnetometer. Accelerometers provide body-frame measurements of the apparent acceleration (all non-gravitational forces per unit mass). Gyroscopes provide the $3$-axis body-frame angular rate while magnetometers measure the earth magnetic field which is assumed constant and known in the inertial frame. Attitude information can be extracted from the IMU-based vector observations (accelerometer and magnetometer measurements) by assuming ``negligible" acceleration of the vehicle using static attitude reconstruction \cite{Markley1988} or more advanced complementary filtering techniques where the gyroscopic measurements are used to complement the vector measurements \cite{tayebi2006attitude,Martin2007,Mahony2008}.

When considering applications with vehicles subject to important accelerations, the above mentioned attitude estimation schemes fail due to the fact that the accelerometer does no longer measures a known inertial vector. To cope with the problem of unknown inertial-frame acceleration, IMU measurements are often complemented with linear velocity measurements. The inertial-frame linear velocity can be, for instance, differentiated to obtain the inertial-frame acceleration which will be then used along with the body-frame acceleration (obtained from the accelerometer) to obtain an attitude information. However, this ad-hoc method is not very desirable in practice due to noise amplification inherited from the approximate derivative operation. A nonlinear \textit{velocity-aided} observer has been proposed in \cite{Martin2008} to estimate both the orientation of the vehicle and the gyro bias vector with local convergence and stability. Following similar lines, velocity-aided attitude estimators have been proposed in \cite{Hua2010,andrew2011velocity-aided} gyro-bias-free case. These observers have the strong property of providing ``meaningful" attitudes on $SO(3)$ with semi-global exponential stability. A velocity-aided attitude and gyro bias estimator has been proposed in \cite{Grip2012} with global exponential stability by relaxing the attitude estimates to lie outside $SO(3)$. Note that in the above mentioned papers, the IMU measurements are complemented with an inertial frame velocity measurements (such as those obtained from a GPS receiver) where some other existing works have considered linear velocity measurements in the body-fixed frame (such as those obtained from Airspeed or Doppler Velocity Log (DVL) sensors) \cite{Bonnabel2006,Bonnabel2008,Dukan2013,Troni2013,Allibert2016,Hua2016}. This paper addresses the attitude estimation problem using IMU and inertial-frame velocity measurements.

The design of a velocity-aided attitude observer on $SO(3)$ with gyro bias estimation and a proven large domain of convergence and stability remains an open problem. In practice, an ad-hoc method to estimate the gyro-bias vector is to add a simple adaptation law (integral action) to the velocity-aided observer of \cite{Hua2010} or \cite{andrew2011velocity-aided} similar to the one which has been used in \cite{Mahony2008} in the case of constant and known inertial vectors. However, no stability proof has been derived until now due to: \textit{i}) the complexity of the proof (even gyro-free case) used in \cite{Hua2010} although albeit relaxed in \cite{andrew2011velocity-aided} and \textit{ii)} the fact that both assumptions of \textit{known} and \textit{constant} inertial vectors used in \cite{Mahony2008} do not hold since the accelerometer measures and \textit{unknown} and \textit{time-varying} inertial vector. In this paper, despite the aforementioned difficulties, we build a velocity-aided attitude observer on $SO(3)$ with gyro bias estimation by augmenting the observer of \cite{andrew2011velocity-aided} with a projection-based adaptive estimation law for the gyro-bias vector. We prove that the proposed attitude observer guarantees semi-global exponential stability. Thereafter, we endorse the proposed attitude estimation scheme with a new hybrid switching mechanism that jumps the attitude states to the region of exponential stability. The velocity-aided hybrid attitude observer is proven to be globally exponentially stable using a rigorous Lyapunov-based proof for hybrid systems.

\section{Background and Preliminaries}\label{sec2}
Throughout the paper, we use $\mathbb{R}$ and $\mathbb{R}_+$ to denote, respectively, the sets of real and nonnegative real numbers. The Euclidean norm of $x\in\mathbb{R}^n$ is defined as $\|x\|=\sqrt{x^\top x}$. For matrices $A, B\in\mathbb{R}^{m\times n}$, their inner product is defined as $\lls A,B\ggs=\textrm{tr}(A^{\top}B)$ and the Frobenius norm of $A$ is $\|A\|_F=\sqrt{\lls A,A\ggs}$. For a square matrix $A\in\mathbb{R}^{n\times n}$, we denote by $\lambda_i^A, \lambda_{\mathrm{min}}^A$, and $\lambda_{\mathrm{max}}^A$ the $i$th, minimum, and maximum eigenvalue of $A$, respectively.
\\
The rigid body attitude evolves on $SO(3) := \{ R \in \mathbb{R}^{3\times 3}|\; \mathrm{det}(R)=1,\; RR^{\top}= I \}$, where $I$ is the three-dimensional identity matrix and $R\in SO(3)$ is called a \textit{rotation matrix}. The group $SO(3)$ has a compact manifold structure with its \textit{tangent spaces} being identified by $T_RSO(3):=\left\{R\Omega\mid\Omega\in\mathfrak{so}(3)\right\},$ where the \textit{Lie algebra} of $SO(3)$, denoted by $\mathfrak{so}(3):=\left\{\Omega\in\mathbb{R}^{3\times 3}\mid\;\Omega^{\top}=-\Omega\right\}$, is the vector space of 3-by-3 skew-symmetric matrices. The map $[\cdot]_\times: \mathbb{R}^3\to\mathfrak{so}(3)$ is defined such that $[x]_\times y=x\times y$, for any $x, y\in\mathbb{R}^3$, where $\times$ is the vector cross-product on $\mathbb{R}^3$. Let $\mathrm{vex}:\mathfrak{so}(3)\to\mathbb{R}^3$ denote the inverse isomorphism of the map $[\cdot]_\times$, such that $\mathrm{vex}([\omega]_\times)=\omega,$ for all $\omega\in\mathbb{R}^3$ and $[\mathrm{vex}(\Omega)]_\times=\Omega,$ for all $\Omega\in\mathfrak{so}(3)$. Defining $\mathbb{P}_a:\mathbb{R}^{3\times 3}\to\mathfrak{so}(3)$ as the projection map on the Lie algebra $\mathfrak{so}(3)$ such that $ \mathbb{P}_a(A):=(A-A^{\top})/2$, we can extend the definition of $\mathrm{vex}$ to $\mathbb{R}^{3\times 3}$ by taking the composition map $\psi := \mathrm{vex}\circ \mathbb{P}_a$ 
        such that, for a $3$-by-$3$ matrix $A:=[a_{ij}]_{i,j = 1,2,3}$, we have
        \begin{equation}\label{psi}
        \psi(A):=\mathrm{vex}\left(\mathbb{P}_a(A)\right)=\frac{1}{2}\left[\begin{array}{c}
        a_{32}-a_{23}\\a_{13}-a_{31}\\a_{21}-a_{12}
        \end{array}
        \right].
        \end{equation}
        Let $R\in SO(3)$ be a rotation matrix, and let $|R|_I\in[0, 1]$ be the normalized Euclidean distance on $SO(3)$ which is given by
     $
        |R|_I^2:=\frac{1}{8}\|I-R\|_F^2=\frac{1}{4}\mathrm{tr}(I-R).
  $
An element $R\in SO(3)$ can be represented as a rotation of angle $\theta\in\mathbb{R}$ around a unit vector axis $u\in\mathbb{S}^2$ using the map $\mathcal{R}_a:\mathbb{R}\times\mathbb{S}^2\to SO(3)$:
\begin{equation}\label{Ra}
\mathcal{R}_a(\theta,u):=e^{\theta[u]_\times}=I+\sin(\theta)[u]_\times+(1-\cos\theta)[u]_\times^2,
\end{equation}
where $e^A$ denotes the matrix exponential of $A$. In this paper, we make use of the framework for dynamical hybrid systems found in \cite{Goebel2006, Goebel2009}. A subset $E\subset\mathbb{R}_{\geq 0}\times\mathbb{N}$ is a  \textit{hybrid time domain}, 
        if it is a union of finitely or infinitely many intervals of the form $[t_j ,t_{j+1}]\times\{j\}$ where $0=t_0\leq t_1\leq t_2\leq...$,  with the last interval being possibly of the form $[t_j ,t_{j+1}]\times\{j\}$ or $[t_j ,\infty)\times\{j\}$. Let $\rightrightarrows$ denote a set-valued mapping. A general model of a hybrid system $\mathcal{H}$ takes the form:
            \begin{equation}\label{Hybrid:general}
            \mathcal{H}\left\{\begin{array}{l}
            \hspace{.25cm}\dot x\in F(x),\hspace{.5cm}x\in C\\
            x^+\in G(x), \hspace{.5cm}x\in D
            \end{array}\right.
            \end{equation}
        where the \textit{flow map}, $F: \mathbb{R}^n\rightrightarrows\mathbb{R}^n$ governs continuous flow of $x\in\mathbb{R}^n$, the \textit{flow set} $C\subset\mathbb{R}^n$ dictates where the continuous flow could occur. The \textit{jump map}, $G: \mathbb{R}^n\rightrightarrows\mathbb{R}^n$, governs discrete jumps of the state $x$, and the \textit{jump set} $D\subset\mathbb{R}^n$ defines where the discrete jumps are permitted. Note that the state $x\in\mathbb{R}^n$ could possibly include both continuous and discrete components. A \textit{hybrid arc} is a function $x: \textrm{dom} \:x\to\mathbb{R}^n$, where $\textrm{dom}\:x$ is a hybrid time domain and, for each fixed $j$, $t\mapsto x(t,j)$ is a locally absolutely continuous function on the interval $I_j=\{t: (t,j)\in\textrm{dom}\:x\}$.
\section{Problem Formulation}\label{section::problem}
Consider the following dynamics of an accelerated rigid body
\begin{equation}\label{model}\left\{
\begin{array}{l}
\dot v=ge_3+Rb_a,\\
\dot{R}=R[\omega]_\times,
\end{array}
\right.
\end{equation}
where $R\in SO(3)$ is the attitude matrix describing the orientation of a body-attached frame with respect to the inertial frame, $\omega\in\mathbb{R}^3$ is the rigid body's angular velocity expressed in the body-attached frame, $v\in\mathbb{R}^3$ is the inertial linear velocity of the rigid body, $g$ is the acceleration due to gravity, $e_3 = [0, 0, 1]^\top$ and $b_a\in\mathbb{R}^3$ is the body-frame ``apparent acceleration", capturing all non-gravitational forces applied to the vehicle expressed in the body frame.

Assume that the following measurements are available :
\begin{itemize}
\item Linear velocity $v$, which may be obtained using a GPS.
\item Magnetometer measurements $b_m$ of the (constant and known) earth magnetic field $r_m$ expressed in the body frame such that $b_m=R^\top r_m$.
\item Accelerometer measurements $b_a$ of the apparent acceleration $r_a:=\dot v-ge_3$ expressed in the body frame such that $b_a=R^\top r_a$.
\item Gyroscope measurements $\omega_y$ of the angular velocity vector $\omega$ such that $\omega_y=\omega+b_\omega$ and $b_\omega\in\mathbb{R}^3$ is a constant gyro bias.
\end{itemize}
Moreover, the following realistic assumptions (constraints) are placed in order to carry out our stability analysis.
\begin{assumption}[Observability condition]\label{assumption::obsv}
There exists a constant $c_0>0$ such that $\|r_m\times r_a(t)\|/\|r_m\|\|r_a(t)\|\geq c_0$ for all $t\geq 0$.
\end{assumption}
\begin{assumption}\label{assumption::bounded_ra}
There exist constants $c_1,c_2,c_3>0$ such that $c_1\leq\|r_a(t)\|\leq c_2$ and $\|\dot r_a(t)\|\leq c_3$ for all $t\geq 0$.
\end{assumption}
\begin{assumption}\label{assumption::bounded_bw}
There exists constants $c_4,c_5>0$ such that $\|\omega(t)\|\leq c_4$ and $\|b_\omega\|\leq c_5$ for all $t\geq 0$.
\end{assumption}
Assumption \ref{assumption::obsv} is a standard (uniform) observability condition in attitude estimation problems. It is guaranteed if the time-varying apparent acceleration $r_a(t)$ is non-vanishing and is \textit{always} not collinear to the constant magnetic field vector $r_m$. Note that $r_a(t)=0$ corresponds  the rigid body being in a free-fall case ($\dot v=ge_3$) which is not likely under normal flight conditions. Assumptions \ref{assumption::bounded_ra} and \ref{assumption::bounded_bw} impose some realistic constraints on the systems trajectory.

Our objective is to design a nonlinear observer that combines all the measurements/information available (as described above) and provide exponentially stable attitude estimates on $SO(3)$ while compensation and estimating the unknown bias vector $b_\omega$. 	

\section{Discussions of Previous Works}
Perhaps the first (invariant) nonlinear attitude observer \textit{with} gyro bias compensation for accelerated vehicles has been proposed in \cite{Martin2008} but only with local stability and convergence analysis. An attitude observer \textit{without} gyro bias compensation has been proposed  in \cite{Hua2010} with proven semi-global exponential stability. It has the following structure
\begin{equation}
\label{Hua2010}
\left\{
\begin{array}{l}
\dot{\hat{v}}=ge_3+\hat Rb_a+k_v\sigma_v,\\
\dot{\hat{R}}=\hat R[\omega_y+k_R\sigma_R]_\times,
\end{array}
\right.
\end{equation}
with $k_v,k_R>0$ and the correction terms $\sigma_v$ and $\sigma_R$ being defined, for $\rho_1,\rho_2>0$, as follows
\begin{align}
\label{sigmav::Hua2010}
\sigma_v&=v-\hat v,\\
\label{sigmaR::Hua2010}
\sigma_R&=\rho_1(b_m\times\hat R^\top r_m)+\rho_2(b_a\times\hat R^\top(v-\hat v)),
\end{align}
Defining the estimation errors $\tilde R=R\hat R^\top$ and $\tilde v=v-\hat v$, the interconnection of the observer \eqref{Hua2010} with the system model \eqref{model} leads to the following closed loop system
\begin{align}
\label{closed::Hua1}
\dot{\tilde R}&=-k_R\rho_1\tilde R[(\tilde R^\top r_m\times r_m)]_\times-g_1(\tilde v),\\
\label{closed::Hua2}
\dot{\tilde v}&=-k_v\tilde v+g_2(\tilde R),
\end{align}
where $g_1(\tilde v)=k_R\rho_2\tilde R[(\tilde R^\top r_a\times\tilde v)]_\times$ and $g_2(\tilde R)=(I-\tilde R^\top)r_a$. The reduced attitude error system \eqref{closed::Hua1} when $\tilde v=0$ corresponds to the closed loop system of a nonlinear complementary filter \cite{Mahony2008} that uses one \textit{single} vector measurement $b_m$ of the magnetic field. Note that the best result one can achieve, using a single vector measurement (without any persistency of excitation condition), is to drive the \textit{reduced} attitude error to zero. Interestingly, however, it is shown in \cite{Hua2010} that for all initial conditions such that $|\tilde R|_I^2<1$, there exists $\underline{k}_v>0$ such that for all $k_v>\underline{k}_v$, the equilibrium point $(\tilde v,\tilde R)=(0,I)$ of the closed loop system \eqref{closed::Hua1}-\eqref{closed::Hua2} is exponentially stable. The dynamic system \eqref{closed::Hua2} of the velocity error $\tilde v$ which feeds the attitude estimation error dynamics through the term $g_1(\tilde v)$ has permitted to drive the \textit{full} attitude error to zero.

Another velocity-aided attitude observer has been introduced in \cite{andrew2011velocity-aided} which takes the same form as \eqref{Hua2010} with the following correction terms
\begin{align}
\label{sigmav::Roberts2011}
\sigma_v&=v-\hat v+\frac{k_R}{k_v^2}\hat R[\sigma_R]_\times b_a,\\
\label{sigmaR::Roberts2011}
\sigma_R&=\rho_1(b_m\times\hat R^\top r_m)+\rho_2(b_a\times\hat R^\top\hat r_a),\\
\label{hatra}
\hat r_a&=k_v(v-\hat v)+\hat Rb_a.
\end{align}
The main difference between the observer of \cite{Hua2010} and that of \cite{andrew2011velocity-aided} is the additional term proportional to $\hat R[\sigma_R]_\times b_a$ in the velocity correction term $\sigma_v$. This has facilitated to obtain an asymptotic ``estimate" of the apparent acceleration vector $r_a(t)$ given by $\hat r_a$ in \eqref{hatra}. In fact, by introducing the estimation error $\tilde r_a=r_a-\hat r_a$, the closed loop system can be written as
\begin{align}
\nonumber
\dot{\tilde R}&=-k_R\tilde R[\rho_1(\tilde R^\top r_m\times r_m)+\rho_2(\tilde R^\top r_a\times r_a)]_\times-g_1(\tilde r_a),\\
\label{closed::andrew}
\dot{\tilde r}_a&=-k_v\tilde r_a+g_2(\tilde R),
\end{align}
where $g_1(\tilde r_a)=k_R\rho_2\tilde R[(\tilde R^\top r_a\times\tilde r_a)]_\times$ and $g_2(\tilde R)=(I-\tilde R^\top)\dot{r}_a$. In contrast to the closed loop system \eqref{closed::Hua1}-\eqref{closed::Hua2} of the observer in \cite{Hua2010}, the reduced attitude error system \eqref{closed::andrew} when $\tilde r_a=0$ corresponds to the closed loop system of a nonlinear complementary filter \cite{Mahony2008} that uses \textit{two} body-frame vector measurements, $r_m$ (the magnetic field) and $r_a$ (the apparent acceleration), which can be shown to be almost globally asymptotically stable (assuming $\tilde r_a=0$). Therefore, the closed loop system \eqref{closed::andrew} can be seen as  an interconnection of two exponentially stable systems. The first (attitude system) is semi-globally exponentially stable and almost globally asymptotically stable and the second ($\tilde r_a$ subsystem) is globally exponentially stable. Moreover, both subsystems can be shown to have an ISS property inside some region \cite{berkane2016design2}. Therefore, under the small gain theorem for ISS systems, the interconnection can be shown to be asymptotically stable as well. This small gain condition is reflected by the condition on the gain $k_v$ and the initial attitude state found in \cite{andrew2011velocity-aided} although derived using a direct Lyapunov-based proof. To put all together, the discussions of this section showed that, at the cost of an additional correction term in the velocity estimation dynamics compared to \cite{Hua2010}, the observer proposed in \cite{andrew2011velocity-aided} is able to estimate the apparent acceleration $r_a(t)$ while also estimating the attitude matrix $R(t)$. Moreover, the closed loop system results in a nice interconnection of two exponentially stable systems. However, both observers fail to provide compensation and/or estimation of the unavoidable gyro bias $b_\omega$ in the angular velocity measurements.
\section{A Velocity-aided Attitude Observer With Gyro-Bias Estimation}
In this section, we provide a solution to the attitude estimation problem of accelerated rigid body systems using GPS and IMU measurements as formulated in Section \ref{section::problem}. We propose the following attitude observer on $SO(3)$ with gyro-bias estimation:
\begin{equation}
\label{observer}
\left\{
\begin{array}{l}
\dot{\hat{v}}=ge_3+\hat Rb_a+k_v\sigma_v,\\
\dot{\hat{R}}=\hat R[\omega_y-\hat b_\omega+k_R\sigma_R]_\times,\\
\dot{\hat{b}}_{\omega}=\mathrm{Proj}\big(\hat b_\omega,-k_b\sigma_R\big),
\end{array}
\right.
\end{equation}
with $k_v,k_R,k_b>0$ and the correction terms $\sigma_R$ and $\sigma_v$ are similar to \eqref{sigmav::Roberts2011}-\eqref{hatra}. The projection function $\mathrm{Proj}$ satisfies the following properties \cite{marino1998robust}:
            \begin{itemize}
            \item [P1.] $\|\hat b_\omega(t)\|\leq c_5,\;\forall t\geq 0$,
            \item [P2.] $(\hat b_\omega-b_\omega)^\top\mathrm{Proj}(\mu,\hat b_\omega)\leq(\hat b_\omega-b_\omega)^\top\mu,$, 
            \item [P3.] $\|\mathrm{Proj}(\mu,\hat b_\omega)\|\leq\|\mu\|$.
            \end{itemize}
 The following theorem is the first main result of the paper.
\begin{theorem}\label{theorem1}
Consider the interconnection of the rigid body dynamics \eqref{model} with the attitude observer \eqref{observer} where Assumptions \ref{assumption::obsv}-\ref{assumption::bounded_bw} are satisfied. For each $0<\varepsilon_R<1$ and for all initial conditions $|\tilde R(0)|_I\leq\varepsilon_R$ and $\|\tilde r_a(0)\|\in\mathbb{R}^3$, there exist gains $\underline{k_v},\underline{k_R}>0$ such that, for all $k_v>\underline{k_v}$ and $k_R\geq\underline{k_R}$, the equilibrium point $(\tilde R,\tilde b_\omega,\tilde r_a)=(I,0,0)$ is exponentially stable.
\end{theorem}
\begin{proof}
The attitude error dynamics are given by
\begin{align}\nonumber
\dot{\tilde R}&=\dot R\hat R^\top-R(\dot{\hat R})^\top=R[\omega]_\times\hat R-R[\omega_y-\hat b_\omega+k_R\sigma_R]_\times\hat R,\\
\label{dR0}
					&=\tilde R[\hat R(-k_R\sigma_R-\tilde b_\omega)]_\times.
\end{align}
On the other hand, it can be shown that $\sigma_R$ defined in \eqref{sigmaR::Roberts2011}-\eqref{hatra} is written as
$
\sigma_R=2\hat R^\top\psi(A\tilde R)-\rho_2(b_a\times\hat R^\top\tilde r_a)),
$
with $A(t)=\rho_1r_mr_m^\top+\rho_2r_a(t)r_a(t)^\top$ where the fact that $\psi(A\tilde R)=\frac{1}{2}\hat R\rho_1(b_m\times\hat R^\top r_m)+\frac{1}{2}\hat R\rho_2(b_a\times\hat R^\top r_a)$ has been used. This results in the following attitude and bias errors dynamics
\begin{align}
\label{dR}
\dot{\tilde R}&=\tilde R[-2k_R\psi(A\tilde R)-\hat R\tilde b_\omega+k_R\rho_2(\tilde Rr_a\times\tilde r_a)]_\times,\\
\label{db}
\dot{\tilde b}_\omega&=\mathrm{Proj}\big(\hat b_\omega,k_b\sigma_R\big).
\end{align}
Moreover, the error dynamics of $\tilde r_a$ are obtained as follows
\begin{align}\nonumber
\dot{\tilde r}_a&=\dot r_a-k_v(\dot v-\dot{\hat v})-(\dot{\tilde R})^\top r_a\\
					  &=-k_v\tilde r_a+(I-\tilde R^\top)\dot r_a
+\hat R[b_a]_\times\tilde b_\omega.
\label{dra}
\end{align}
Note hat all the error signals involved in the closed loop system \eqref{dR}-\eqref{dra} are \textit{a priori} bounded. In fact, the attitude error state $\tilde R$ is bounded by the compactness of $SO(3)$. The bias estimation error is bounded thanks to the projection mechanism. Let $c_b>0$ be an upper bound on $\|\tilde b_\omega(t)\|$ for all $t\geq 0$. Moreover, in view of the fact that $\|I-\tilde R\|_F^2=8|\tilde R|_I^2$ and using Assumption 1, it follows that
\begin{align}\label{dra2}
\dot{\tilde r}_a&\leq-k_v\tilde r_a+(\sqrt{8}c_3+c_2c_b),
\end{align}
which shows that $\tilde r_a$ cannot grow unbounded due to the presence of the negative term $-k_v\tilde r_a$. Our next goal is to prove that for all $0<\epsilon_R<1$ and under some conditions on $k_v$ and $k_R$, the set $\mho(\varepsilon_R)=\{\tilde R\in SO(3)\mid |\tilde R|_I\leq\epsilon_R\}$ is forward invariant. Integrating \eqref{dra2} and using the comparison lemma, one obtains
\begin{align}\label{ra2}
\|\tilde r_a(t)\|\leq e^{-k_vt}\|\tilde r_a(0)\|+\frac{c_a}{k_v}(1- e^{-k_vt}),\quad\forall t\geq 0,
\end{align}
with $c_a=\sqrt{8}c_3+c_2c_b$. It can be verified from \eqref{ra2} that the upper bound on $\|\tilde r_a(t)\|$ is either increasing from $\|\tilde r_a(0)\|$ to $c_a/k_v$ (if $\|\tilde r_a(0)\|\leq c_a/k_v$), or decreasing from $\|\tilde r_a(0)\|$ to $c_a/k_v$ (if $\|\tilde r_a(0)\|\geq c_a/k_v$). Therefore, in all cases, for all $\varepsilon_a>0$ and choosing $k_v>c_a/(\|\tilde r_a(0)\|+\varepsilon_a)$, one has $\|\tilde r_a(t)\|\leq \|\tilde r_a(0)\|+\varepsilon_a$ for all $t\geq 0$. Under this condition, let us compute the minimum time $t_R$ necessary for $\tilde R(t)$ to go, in the worst case senario, outside the set $\mho(\varepsilon_R)$ (starting from $\tilde R(0)\in\mho(\varepsilon_R)$). The time derivative of $|\tilde R|_I^2=\mathrm{tr}(I-\tilde R)/4$, in view of \eqref{dR}, satisfies
\begin{align}
\nonumber
&\frac{d}{dt}|\tilde R|_I^2=-\mathrm{tr}(\dot{R})/4\\
\nonumber
&=-\mathrm{tr}(\mathbb{P}_a(\tilde R)[-2k_R\psi(A\tilde R)-\hat R\tilde b_\omega+k_R\rho_2(\tilde Rr_a\times\tilde r_a)]_\times)/4\\
\nonumber
&=-k_R\psi(\tilde R)^\top\psi(A\tilde R)-\frac{1}{2}\psi(\tilde R)^\top(\hat R\tilde b_\omega-k_R\rho_2(\tilde Rr_a\times\tilde r_a))\\
&=-4k_R\lambda_{\min}^{\bar A}|\tilde R|_I^2(1-|\tilde R|_I^2)+|\tilde R|_I\|\tilde b_\omega\|+k_R\rho_2c_2|\tilde R|_I\|\tilde r_a\|
									  \label{dR2}
\end{align}
with $\bar A=\frac{1}{2}(\mathrm{tr}(A)-A)$, where we used the following facts (see \cite{berkane2016design} and identities therein) $\mathrm{tr}([u]_\times[v]_\times)=-2u^\top v$, $\mathbb{P}_a(\tilde R)=[\psi(\tilde R)]_\times, \psi(\tilde R)^\top\psi(A\tilde R)=\psi(\tilde R)^\top\bar A\psi(\tilde R)$ and $\|\psi(\tilde R)\|^2=4|\tilde R|_I^2(1-|\tilde R|_I^2)$. It follows that
\begin{align*}
\frac{d}{dt}|\tilde R|_I^2&\leq|\tilde R|_I\|\tilde b_\omega\|+k_R\rho_2c_2|\tilde R|_I\|\tilde r_a\|,\\
									&\leq c_b+k_R\rho_2c_2(\|\tilde r_a(0)\|+\varepsilon_a).
\end{align*}
In view of the above inequality on the velocity of $|\tilde R|_I^2$, it can be deduced that the minimum time necessary for $\tilde R(t)$ to go outside the set $\mho(\varepsilon_R)$ satisfies
$$
t_R\geq\underline{t}_R=\frac{\varepsilon_R^2-|\tilde R(0)|_I^2}{c_b+k_R\rho_2c_1(\|\tilde r_a(0)\|+\varepsilon_a)}.
$$
Since we have a knowledge about the minimum time necessary for the attitude error to go outside the set $\mho(\varepsilon_R)$, it is possible to prevent such a scenario by imposing some (high gain) conditions on the gains $k_v$ and $k_R$. First, we start by finding a minimum gain on $k_v$ such that $\|\tilde r_a(t)\|\leq B_a/k_R$, for some $B_a>0$, for all $t\geq t_a$, and such that $t_a\geq t_R\geq 0$. Assume that $k_v\geq k_v^*>c_ak_R/B_r$ for some $k_v^*>0$. Note that if $\|\tilde r_a(0)\|\leq c_a/k_v$ then $\|\tilde r_a(t)\|\leq c_a/k_v<B_r/k_R$ for all $t\geq 0$. Otherwise, the minimum time $t_a$ necessary to enter the ball $\|\tilde r_a(t)\|\leq\frac{B_r}{k_R}$ satisfies
\begin{align}
t_a\leq\bar{t}_a=\frac{1}{k_v^*}\ln\left(\frac{\|\tilde r_a(0)\|-c_a/k_v^*}{B_r/k_R-c_a/k_v^*}\right).
\end{align}
Note that the value of $k_v^*$ can be arbitrary increased to make $\bar{t}_a$ arbitrary small. Let $k_v^*$ be chosen such that $\bar{t}_a\leq\underline{t}_R$. Hence, it this case, it is true that $\|\tilde r_a(t)\|\leq B_r/k_R$ for all $t\geq t_a$. Therefore, it follows from \eqref{dR2} that for all $t\geq t_a$ one has
\begin{align*}
\frac{d|\tilde R(t)|_I^2}{dt}&\leq-4k_R\lambda_{\min}^{\bar A}|\tilde R(t)|_I^2(1-|\tilde R(t)|_I^2)+c_b+\rho_2c_2B_a.
\end{align*}
Note that the matrix $\bar A$ is positive definite in view of Assumption \ref{assumption::obsv}. In fact, it can be easily verified that $\bar A=\frac{1}{2}(\mathrm{tr}(A)-A)=-\frac{1}{2}\rho_1[r_m]_\times^2-\frac{1}{2}\rho_2[r_a]_\times^2$ which is positive definite if $r_m$ and $r_a(t)$ are non-collinear for all times. Now assume that $|\tilde R(t)|_I=\varepsilon_R$ and $k_R>(c_b+\rho_2c_2B_a)/(4\lambda_{\min}^{\bar A}\varepsilon_R^2(1-\varepsilon_R^2))$ then one has
\begin{align*}
\frac{d}{dt}|\tilde R(t)|_I^2&\leq-4k_R\lambda_{\min}^{\bar A}\varepsilon_R^2(1-\varepsilon_R^2)+c_b+\rho_2c_2B_a<0.
\end{align*}
This implies that $|\tilde R(t)|_I$ is strictly decreasing whenever $|\tilde R(t)|_I=\varepsilon_R$. It follows from the continuity of the solution that $\tilde R(t)$ will never move outside the ball $\mho(\varepsilon_R)$ for all $t\geq t_a$. Recall also that $|\tilde R(t)|_I\leq\varepsilon_R$ for all $t\leq t_R$. Consequently, since $t_a\leq\bar t_a\leq\underline{t}_R\leq t_R$, one concludes that $|\tilde R(t)|_I\leq\varepsilon_R$ for all $t\geq 0$ under the following gain conditions
\begin{align}\label{kR1}
k_R&>\frac{c_b+\rho_2c_2B_a}{4\lambda_{\min}^{\bar A}\varepsilon_R^2(1-\varepsilon_R^2)},&\forall B_a>0.\\
k_v&>\max\left(\frac{c_a}{\|\tilde r_a(0)\|+\varepsilon_a},k_v^*\right),&\forall \varepsilon_a>0,
\end{align}
which implies that the set $\mho(\varepsilon_R)$ is forward invariant. Now, we are ready to prove the exponential stability. Consider the following Lyapunov function candidate
\begin{align}
\label{V}
V=|\tilde R|_I^2+\frac{\mu k_R}{2k_b}\tilde b_\omega^\top\tilde b_\omega+\mu\tilde b_\omega^\top\hat R^\top\psi(\tilde R)+\frac{1}{2}\tilde r_a^\top\tilde r_a,
\end{align}
where $\mu$ is some positive scalar. Using the fact that $\|\psi(\tilde R)\|\leq 2|\tilde R|_I$ and letting $z:=[z_1,z_2,z_3]^\top=[|\tilde R|_I,\|\tilde b_\omega\|,\|\tilde r_a\|]^\top$, it can be checked that $V$ satisfies the quadratic inequality $z^\top P_1z\leq V\leq z^\top P_2 z$ where the matrices $P_1$ and $P_2$ are given by
\begin{align*}
P_1=\begin{bmatrix}
1&-\mu&0\\
-\mu&\frac{\mu k_R}{2k_b}&0\\
0&0&\frac{1}{2}
\end{bmatrix},\quad P_2=\begin{bmatrix}
1&\mu&0\\
\mu&\frac{\mu k_R}{2k_b}&0\\
0&0&\frac{1}{2}
\end{bmatrix}.
\end{align*}
Let us compute the time derivative of the cross term $\mathfrak{X}=\tilde b_\omega^\top\hat R^\top\psi(\tilde R)$ along the trajectories of the closed-loop system. Using \cite[Lemma 1]{berkane2016design} and in view of \eqref{dR0} one obtains $\dot{\psi}(\tilde R)=E(\tilde R)(-\hat R(\tilde b_\omega+k_R\sigma_R))$ with $E(\tilde R)=\frac{1}{2}(\mathrm{tr}(\tilde R)I-\tilde R)$. Also, it can be checked that the following properties for $E(\tilde R)$ hold:
\begin{align}
&x^\top(I-E(\tilde R))x\leq 2|\tilde R|_I^2\|x\|^2,\\
&x^\top(I-E(\tilde R))y\leq (2|\tilde R|_I^2+\sqrt{2}|\tilde R|_I)\|x\|\|y\|,
\end{align}
for all $x,y\in\mathbb{R}^3$ and $\tilde R\in SO(3)$. Moreover, one has $\|\sigma_R\|\leq2\|\psi(A\tilde R)\|+\rho_2\|r_a\|\|\tilde r_a\|\leq 4\lambda_{\max}^{\bar A}|\tilde R|_I+\rho_2c_2\|\tilde r_a\|$. Consequently, it follows that
{\small
\begin{align*}\nonumber
				\dot{\mathfrak{X}}
                &=\tilde b_\omega^\top\hat R^\top E(\tilde R)\left(-\hat R(\tilde b_\omega+k_R\sigma_R)\right)-\tilde b_\omega^\top[\omega+\tilde b_\omega+k_R\sigma_R]_\times\\
                &\hat R^\top\psi(\tilde R)+\mathrm{Proj}\big(\hat b_\omega,k_b\sigma_R\big)^\top\hat R^\top\psi(\tilde R)\nonumber\\
                \nonumber
                    &\leq-\|\tilde b_\omega\|^2+\tilde b_\omega^\top\hat R^\top (I-E(\tilde R))\hat R\tilde b_\omega+k_R\tilde b_\omega^\top\hat R^\top (I-E(\tilde R))\hat R\sigma_R\\
                    &-k_R\tilde  b_\omega^\top\sigma_R+c_\omega\|\tilde b_\omega\|\big\|\psi(\tilde R)\big\|+(k_Rc_b+k_b)\|\sigma_R\|\big\|\psi(\tilde R)\big\|\\
                    \nonumber
                     &\leq-\|\tilde b_\omega\|^2-k_R\tilde  b_\omega^\top\sigma_R+2c_b^2|\tilde R|_I^2+k_Rc_b\big(2|\tilde R|_I^2+\sqrt{2}|\tilde R|_I\big)\|\sigma_R\|\\
                     &+2(k_b+k_Rc_b)|\tilde R|_I\|\sigma_R\|+2c_\omega\|\tilde b_\omega\||\tilde R|_I\nonumber\\
               &\leq-\|\tilde b_\omega\|^2-k_R\tilde  b_\omega^\top\sigma_R+(\alpha_1+k_R\alpha_2)|\tilde R|_I^2\\
               &+(\alpha_3+k_R\alpha_4)|\tilde R|_I\|\tilde r_a\|+2c_\omega\|\tilde b_\omega\||\tilde R|_I,
            \end{align*}
            }
such that $\alpha_1=2c_b^2+8\lambda_{\max}^{\bar A} k_b, \alpha_2=4\lambda_{\max}^{\bar A} c_b(4+\sqrt{2}), \alpha_3=2\rho_2c_2k_b$ and $\alpha_4=\rho_2c_2c_b(4+\sqrt{2})$.
Recall also that
\begin{align*}
\frac{1}{2}\frac{d}{dt}\|\tilde b_\omega\|^2&\leq \tilde b_\omega^\top\mathrm{Proj}\big(\hat b_\omega,k_b\sigma_R\big)\leq k_b\tilde b_\omega^\top\sigma_R\\
\frac{1}{2}\frac{d}{dt}\|\tilde r_a\|^2&\leq-k_v\|\tilde r_a\|^2+\sqrt{8}c_3|\tilde R|\|\tilde r_a\|+c_2\|\tilde b_\omega\|\|\tilde r_a\|.
\end{align*}
Consequently, in view of the above results, the time derivative of $V$ along the trajectories of the closed loop system satisfies
\begin{align}
\label{dV}
\dot V&\leq-z_{12}^\top P_{12}z_{12}-z_{13}^\top P_{13}z_{13}-z_{23}^\top P_{23}z_{23},
\end{align}
where $z_{ij}=[z_i,z_j]^\top$ and the matrices $P_{ij}$ are
\begin{align*}
P_{12}&=\begin{bmatrix}
k_R\big(2\lambda_{\min}^{\bar A}(1-\varepsilon_R^2)-\mu\alpha_2\big)-\mu\alpha_1&-(\frac{1}{2}+c_\omega\mu)\\
*&\frac{\mu}{2}
\end{bmatrix},\\
P_{13}&=\begin{bmatrix}
2k_R\lambda_{\min}^{\bar A}(1-\varepsilon_R^2)&-(\frac{k_R\rho_2c_2+\sqrt{8}c_3+\mu(\alpha_3+k_R\alpha_4)}{2})\\
*&\frac{k_v}{2}
\end{bmatrix},\\
P_{23}&=\begin{bmatrix}
\frac{\mu}{2}&-\frac{c_2}{2}\\
-\frac{c_2}{2}&\frac{k_v}{2}
\end{bmatrix}.
\end{align*}
Now, if we pick $\mu>0$ such that $\mu<\lambda_{\min}^{\bar A}(1-\varepsilon_R^2)/\alpha_2$ and choose the gains $k_R$ and $k_v$ such that
\begin{align*}
k_R&>\max\left\{2\mu k_b,\frac{2\alpha_1\mu^2+(1+2c_\omega\mu)^2}{2\mu\lambda_{\min}^{\bar A}(1-\varepsilon_R^2)}\right\},\\
k_v&>\max\left\{\frac{c_1^2}{\mu}, \frac{2(\frac{k_R\rho_2c_1}{2}+\sqrt{2}c_2+\frac{\mu(\alpha_3+k_R\alpha_4)}{2})^2}{k_R\lambda_{\min}^{\bar A}(1-\varepsilon_R^2)}\right\},
\end{align*}
then matrices $P_1, P_2, P_{12}, P_{13}$ and $P_{23}$ are all positive definite. The exponential stability immediately follows.
\end{proof}
\section{Hybrid Velocity-Aided Attitude Observer With Gyro-Bias Estimation}
The exponential stability of the proposed attitude observer \eqref{observer} is best described as semi global. Note that, for large initial conditions such that $|\tilde R(0)|_I^2\to 1$, the conditions of Theorem \ref{theorem1} require high gains for $k_R$ and $k_v$ which tend to infinity as the attitude error gets closer to $180^\circ$. Although, simulation results suggest that this might not be an issue in practice and that the derived high-gain conditions on the observer gains are only conservative, it is desirable from a theoretical point of view to remove these conditions. Hybrid observers have been proposed recently in \cite{lee2015observer,berkaneCDC2016observer} to overcome the topological obstruction for global attractivity on $SO(3)$. Our hybrid observer proposed in \cite{berkaneCDC2016observer} which is an extension to the nonlinear complementary filter \cite{Mahony2008} guarantees global exponential stability when using vector measurements of constant and known inertial vectors. In this section, we adopt a different approach to design a hybrid observer that is able to deal with ``time-varying" vector measurements, which is suitable for the problem dealt with in the present work.

First, let us define the following function
\begin{multline}
\Phi_0(\hat R,b_m,b_a,r_a)=3-\frac{r_m^\top\hat Rb_m}{\|b_m\|^2}-\frac{(r_m\times r_a)\hat R(b_m\times b_a)}{\|b_m\times b_a\|^2}\\-\frac{(r_m\times r_m\times r_a)\hat R(b_m\times b_m\times b_a)}{\|b_m\times b_m\times b_a\|^2}.
\end{multline}
For simplicity, if no argument is indicated for $\Phi_0$ then it should be understood that $\Phi_0\equiv\Phi_0(\hat R,b_m,b_a,r_a)$. Consider the following ``reset" rule for the attitude estimate $\hat R$
\begin{align}
\label{flow}
&\left\{
\begin{array}{l}
\dot{\hat{v}}=ge_3+\hat Rb_a+k_v\sigma_v,\\
\dot{\hat{R}}=\hat R[\omega_y-\hat b_\omega+k_R\sigma_R]_\times,\\
\dot{\hat{b}}_{\omega}=\mathrm{Proj}\big(\hat b_\omega,-k_b\sigma_R\big),
\end{array}
\right.&(\hat v,\hat R,\hat b_\omega)\in\mathcal{F},\\
&\left\{
\begin{array}{l}
\hat v^+=\hat v,\\
\hat R^+=\mathcal{R}_a(\pi,u)\hat R,\\
\hat b_\omega^+=\hat b_\omega,
\end{array}
\right.&(\hat v,\hat R,\hat b_\omega)\in\mathcal{J},
\label{jump}
\end{align}
where the flow set $\mathcal{F}$ and jump set $\mathcal{J}$ are defined as follows
\begin{align}
\mathcal{F}&=\{(\hat v,\hat R,\hat b_\omega):\;\Phi(\hat v,\hat R)\leq\delta\},\\
\mathcal{J}&=\{(\hat v,\hat R,\hat b_\omega):\;\Phi(\hat v,\hat R)\geq\delta\},
\end{align}
and $\Phi(\hat v,\hat R)=\Phi_0(\hat R,b_m,b_a,\hat Rb_a+k_v(v-\hat v))$ which is a measurable quantity. The correction terms $\sigma_v$ and $\sigma_R$ are similar to \eqref{sigmav::Roberts2011} and \eqref{sigmaR::Roberts2011}-\eqref{hatra}, respectively. The unit vector $u\in\mathbb{S}^2$ that appears in \eqref{jump} is defined as
\begin{align}
\label{u}
u=\mathrm{arg}\min_{i\in\{1,2,3\}}\Phi(\hat v,\mathcal{R}(\pi,u_i)\hat R),
\end{align}
where $\{u_1,u_2,u_3\}$ is any orthonormal basis of $\mathbb{R}^3$. We state our second main result of the paper.
\begin{theorem}\label{theorem2}
Consider the attitude kinematics with the hybrid attitude observer \eqref{flow}-\eqref{u} where Assumptions \ref{assumption::obsv}-\ref{assumption::bounded_bw} are satisfied. Assume that $\delta,\alpha>0$ such that $3+5\alpha/2<\delta<4-\alpha$ and $\alpha<2/7$. Then, there exist gains $\underline{k_v},\underline{k_R}>0$ such that, for all $k_v>\underline{k_v}$ and $k_R\geq\underline{k_R}$, the equilibrium point $(\tilde R,\tilde b_\omega,\tilde r_a)=(I,0,0)$ is globally exponentially stable.
\end{theorem}

Before proceeding with the proof of the theorem, some remarks are in order. The innovative idea of the hybrid observer \eqref{flow}-\eqref{u} is to employ discrete transitions of the estimated attitude state $\hat R$ rather than transitions in the observer's correction term as in previous synergistic hybrid techniques \cite{lee2015observer,berkaneCDC2016observer}. This allows to deal with the problem of time-varying inertial as is the case of the problem at hand. In fact, rotating the attitude estimate by an angle of $180^\circ$ (which results in a rotation of $180^\circ$ for the attitude error $\tilde R=R\hat R^\top$ as well) in some specific direction $u$ whenever the current attitude error is close to $180^\circ$, results in a decrease in the attitude error. However, since measurements of the attitude error is not available to check how ``close" the error is to $180^\circ$, we use the measurable cost $\Phi(\hat v,\hat R)$ as a criterion. In fact, under the conditions of Theorem \ref{theorem2}, it will be shown in the proof of the theorem that when the attitude error is close to $180^\circ$, the cost $\Phi(\hat v,\hat R)$ exceeds certain pre-defined constant threshold $\delta$ which forces the observer \eqref{flow}-\eqref{u} to jump its state and therefore reducing the estimation error.

\begin{proof}
Let us show that there exist $0<\epsilon_R<1$ and an initial finite time $t_0\geq 0$ such that $|\tilde R(t)|_I^2\leq\epsilon_R$ during the flows of $\mathcal{F}$ for all $t\geq t_0$. In perfect conditions (noise-free), it can be verified that $\Phi_0(\hat R,b_m,b_a,r_a)=\mathrm{tr}(I-\tilde R)=4|\tilde R|_I^2$. Using the fact that $\hat Rb_a+k_v(v-\hat v)=r_a-\tilde r_a$, it can be verified that
\begin{multline*}
\Phi=\Phi_0+\frac{(r_m\times\tilde r_a)\hat R(b_m\times b_a)}{\|b_m\times b_a\|^2}+\\\frac{(r_m\times r_m\times\tilde r_a)\hat R(b_m\times b_m\times b_a)}{\|b_m\times b_m\times b_a\|^2}.
\end{multline*}
Therefore, in view of Assumptions \eqref{assumption::obsv}-\eqref{assumption::bounded_ra}, it follows that
\begin{align}
\Phi_0-\frac{2\|\tilde r_a\|}{c_1c_0^2}\leq\Phi\leq\Phi_0+\frac{2\|\tilde r_a\|}{c_1c_0^2}.
\end{align}
Assume that the gain $k_v$ is chosen such that $k_v>2c_a/\alpha c_1c_0^2$ for some $\alpha>0$. Hence, there exists a finite time $t_0\geq 0$ such that $\|\tilde r_a(t)\|\leq\alpha c_1c_0^2/2$ for all $t\geq t_0$. Therefore, the following holds
\begin{align}\label{ineq_Phi}
\Phi_0-\alpha\leq\Phi\leq\Phi_0+\alpha,\quad t\geq t_0.
\end{align}
Let $\tilde R\in\mathcal{R}_a(\pi,\mathbb{S}^2)$ which implies that $\Phi_0=4|\tilde R|_I^2=4$. Hence, making use of \eqref{ineq_Phi}, it follows that $\Phi\geq\Phi_0-\alpha=4-\alpha>\delta$. Consequently, one has $(\hat v,\hat R,\hat b_\omega)\in\mathcal{J}$. Since $|\tilde R|_I^2$ can not be equal $1$ inside the flow set $\mathcal{F}$, one concludes that there exist $0<\epsilon_R<1$ such that $|\tilde R(t)|_I\leq \epsilon_R<1$ for all $t\geq t_0$. We have just showed that the set $\mho(\epsilon_R)$ is forward invariant (starting from time $t_0\geq 0$) without requiring a condition on the initial states. The proof of exponential convergence from all initial conditions, during the flows of $\mathcal{F}$, follows the same lines as \eqref{V}-\eqref{dV} and here omitted for space limitation. It remains to show that the Lyapunov function $V$ used in \eqref{V} is strictly decreasing during the jumps of the hybrid observer \eqref{jump}. First, using the fact that $\mathcal{R}_a(\pi,u_i)=-I+2u_iu_i^\top$, one has
\begin{align*}
\sum_{i=1}^3\mathrm{tr}(I-\tilde R\mathcal{R}(\pi,u_i))&=3\mathrm{tr}(I+\tilde R)-2\sum_{i=1}^3\mathrm{tr}(\tilde Ru_iu_i^\top),\\
																	   &=12-\mathrm{tr}(I-\tilde R),
\end{align*}
where we have used the fact that $\{u_1,u_2,u_3\}$ is an orthonormal basis and hence $\sum_{i=1}^3u_iu_i^\top=1$. Moreover, in view of \eqref{u} and \eqref{ineq_Phi} and using the fact that $\Phi^+=\min_{i\in\{1,2,3\}}\Phi(\hat v,\mathcal{R}(\pi,u_i)\hat R)$, one obtains
\begin{align*}
\Phi^+&\leq\frac{1}{3}\sum_{i=1}^3\Phi(\hat v,\mathcal{R}(\pi,u_i)\hat R)\\
		  &\leq\frac{1}{3}\left(\sum_{i=1}^3\mathrm{tr}(I-\tilde R\mathcal{R}(\pi,u_i))+3\alpha\right)=4-\frac{4}{3}|\tilde R|_I^2+\alpha.
\end{align*}
Consequently, it follows using \eqref{ineq_Phi} and $|\tilde R^+|_I^2=\Phi_0^+/4$ that
\begin{align*}
|\tilde R^+|_I^2-|\tilde R|_I^2&\leq\frac{1}{4}\left(\Phi^++\alpha\right)-|\tilde R|_I^2\leq1-\frac{4}{3}|\tilde R|_I^2+\frac{\alpha}{2},\\
											 &\leq1-\frac{(\delta-\alpha)}{3}+\frac{\alpha}{2}=\frac{1}{3}(3+5\alpha/2-\delta).
\end{align*}
where we have also used the fact that $|\tilde R|_I^2=\Phi_0/4\geq(\Phi-\alpha)/4\geq(\delta-\alpha)/4$. Therefore, by letting the constant $\mu$ in \eqref{V} to satisfy $0<\mu\leq(\delta-5\alpha/2-3)/24c_b$, one deduces
\begin{align*}
&V^+-V=|\tilde R^+|_I^2-|\tilde R|_I^2+\mu\tilde b_\omega^\top\hat R^\top(\psi(\tilde R^+)-\psi(\tilde R))\\
		  &\leq\frac{1}{3}(3+5\alpha/2-\delta)+4\mu c_b\leq\frac{1}{6}(3+5\alpha/2-\delta)<0.
\end{align*}
The proof is complete by invoking \cite[Theorem 1]{teel2013lyapunov}.
\end{proof}

\section{Simulation results}
We provide simulation results that demonstrate the effectiveness of the proposed velocity-aided estimation schemes (hybrid and non-hybrid) with bias compensation. We consider a rigid body system evolving according to \eqref{model} with the following linear velocity and angular velocity vectors

{\small
\begin{align*}
v(t)=\begin{bmatrix}
2\cos(0.5t+0.5)\\
3.75\cos(1.25t+0.5)\\
0.5\cos(0.5t+0.5)
\end{bmatrix},\;\;\;
\omega(t)=\begin{bmatrix}
\sin(0.1t+\pi)\\
0.5\sin(0.2t)\\
0.1\sin(0.3t+\pi/3)
\end{bmatrix}.
\end{align*}
}
The true attitude is initialized at $R(0)=\mathcal{R}_a(\pi,[0,1,0]^\top)$ and the estimated attitude at $\hat R(0)=I$. We consider gyroscopic measurements of the angular velocity vector with a constant bias $b_\omega=[5,5,5]^\top\mathrm{(deg/s)}$. Also we consider an IMU equipped with a magnetometer and an accelerometer providing in the dody-frame, respectively,  measurements of the earth's magnetic field $r_m=[0.18, 0, 0.54]^\top$ and the apparent acceleration (unknown in the inertial frame). We implement the observer \eqref{observer} with the gains $k_v=\rho_1=\rho_2=1$, $k_R=2$ and $k_v=3$. The hybrid observer \eqref{flow}-\eqref{u} is also implemented with the same gains and parameters $u_i=e_i, i=1,2,3$ and $\delta=3.6$. The hybrid observer has switched at around $1.42\mathrm{s}$ to reduce the attitude error as shown in Figure \ref{figure::theta}. Although both observers are successful in estimating the attitude, the gyro bias and the unknown acceleration of the vehicle, this simulation showed that the hybrid observer exhibits a faster transient response.
\begin{figure}[h!]
\centering
\includegraphics[scale=0.35]{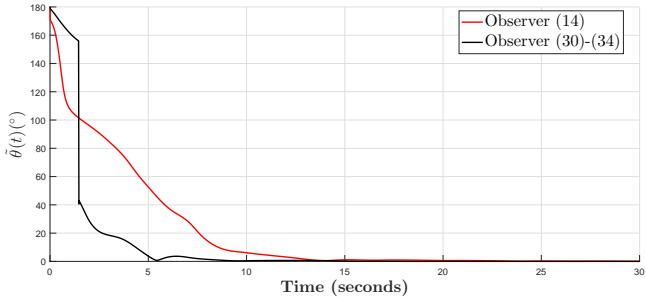}
\caption{Attitude estimation error angle in degrees.}
\label{figure::theta}
\end{figure}
\begin{figure}[h!]
\centering
\includegraphics[scale=0.35]{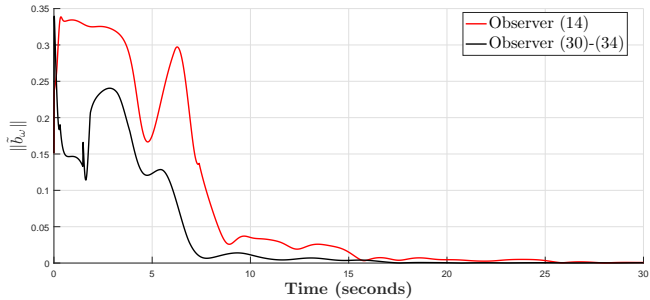}
\caption{Bias estimation error in rad/s.}
\label{figure::bias}
\end{figure}
\begin{figure}[h!]
\centering
\includegraphics[scale=0.35]{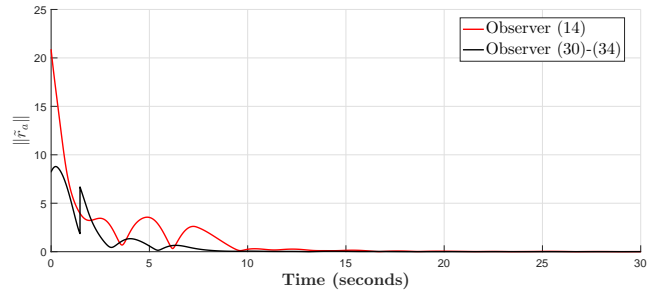}
\caption{Apparent acceleration estimation error in m/s.}
\label{figure::bias}
\end{figure}

\section{Conclusion}
We extended the velocity-aided attitude observer of \cite{andrew2011velocity-aided} with a projection-based gyro-bias estimation scheme. We proved that the origin of the closed-loop system is semi-globally exponentially stable. To the best of our knowledge, this is the first result providing semi-global exponential stability results for a velocity-aided attitude observer \textit{with gyro-bias estimation}. This extension is far from being trivial due to the fact that one of the reference vectors, namely $r_a$, is unknown and time-varying. Moreover, we proposed a new hybrid attitude observer scheme to enlarge the region of attraction of the observer, leading to global exponential stability results. The hybrid observer jumps its state to an attitude estimate which is guaranteed to be \textit{closer} to the true attitude after each jump. The performance of the proposed estimation schemes is illustrated by some simulation results.
\bibliographystyle{IEEETran}
\bibliography{IEEEabrv,Hybrid}
\end{document}